\documentclass{amsart}

\usepackage[utf8]{inputenc}
\usepackage{amsmath}
\usepackage{amsthm}
\usepackage{amsfonts}
\usepackage{amssymb}
\usepackage{color}
\usepackage{leftidx}
\usepackage{tikz}
\usetikzlibrary{matrix,arrows,decorations.pathmorphing}
\usepackage{wrapfig}
\usepackage{tikz-cd}
\usepackage{graphicx}

\newtheorem{theorem}{Theorem}[section]
\newtheorem{lemma}[theorem]{Lemma}

\theoremstyle{definition}

\theoremstyle{remark}

\numberwithin{equation}{section}

\newcommand{\Ker}{\text{Ker}\ }
\newcommand{\Span}{\text{span}}

\newcommand{\ad}{\text{ad}}
\newcommand{\Exp}{\text{exp}}

\newcommand{\red}{\text{red}}
\newcommand{\can}{\text{can}}
\newcommand{\Top}{\text{top}}
\newcommand{\Lie}{\text{Lie}}

\theoremstyle{remark}

\theoremstyle{plain}
\newtheorem{thm}{Theorem}[section]
\newtheorem{cor}[thm]{{\bf Corollary}}
\newtheorem{lem}[thm]{Lemma}
\newtheorem{claim}[thm]{Claim}
\newtheorem{prop}[thm]{Proposition}

\begin{document}

\title{Topological Entropy of Left-Invariant Magnetic Flows on 2-Step Nilmanifolds}

\author{Jonathan Epstein}
\address{Mathematics Department, Dartmouth College, Hanover, NH 03755}
\email{jonathan.m.epstein.gr@dartmouth.edu}

\date{October 22, 2015}

\keywords{Differential geometry, topological entropy, magnetic flows, nilmanifolds}

\begin{abstract}
In this paper, we consider magnetic flows on 2-step nilmanifolds $M = \Gamma \backslash G$, where the Riemannian metric $g$ and the magnetic field $\sigma$ are left-invariant. Our first result is that when $\sigma$ represents a rational cohomology class and its restriction to $\mathfrak{g} = T_eG$ vanishes on the derived algebra, then the associated magnetic flow has zero topological entropy.  In particular, this is the case when $\sigma$ represents a rational cohomology class and is exact.  Our second result is the construction of a magnetic field on a 2-step nilmanifold that has positive topological entropy for arbitrarily high energy levels.
\end{abstract}

\maketitle

\section{Introduction} \label{sec:Intro}
Let $M$ be a closed $n$-dimensional manifold endowed with a $C^\infty$ Riemannian metric $g$ and let $\sigma$ be any closed 2-form on $M$.  The data $(M,g,\sigma)$ now determine a magnetic flow on the cotangent bundle $T^*M$, as follows.  Let $\omega_\can$ be the canonical symplectic structure on $T^*M$, and $\pi : T^*M \to M$ the canonical projection.  Then $\omega_\can + \pi^* \sigma$ is a closed nondegenerate 2-form on $T^*M$.  In other words, $\sigma$ determines a ``twisted'' symplectic structure on the cotangent bundle $T^*M$.  The magnetic flow $\phi_t$ is the Hamiltonian flow on the symplectic manifold $(T^*M, \omega_\can + \pi^* \sigma)$ of the metric Hamiltonian $H: T^*M \to \mathbb{R}$ defined by
\begin{align} \label{eq:metric hamiltonian def}
H(x,\lambda) = \frac{1}{2}g_x(\lambda, \lambda) = \frac{1}{2}|\lambda|^2
\end{align}
where, by abuse of notation, $g$ also denotes the metric induced in the cotangent bundle.  The terminology comes from the fact that $\phi_t$ models the motion of a particle of unit mass and unit charge under the effect of a magnetic field. When $\sigma$ is exact, $\sigma = d\tau$ for some 1-form $\tau$ on $M$, then $\phi_t$ is the same as the flow of the Hamiltonian system $(T^*M, \omega_\can, H_1)$ where $H_1 : T^*M \to \mathbb{R}$ is the function $H_1(x,\lambda) = (1/2)|\lambda + \tau_x|^2$.  When $\sigma \equiv 0$, we obtain the geodesic flow.  Since $H$ is preserved by the Hamiltonian flow, the magnetic flow leaves the energy hypersurfaces $H^{-1}(c) \subset T^*M$ invariant.  However, unlike the geodesic flow, the magnetic flow on two different energy levels are not reparameterizations of each other.  Magnetic flows were first considered by Arnold in \cite{arnold} and by Anosov and Sinai in \cite{anosov_sinai}.

\bigskip

\noindent The topological entropy of a dynamical system is a measure of the global orbit structure complexity.  Roughly speaking, it is the exponential growth rate of the number of orbit segments that can be distinguished with arbitrarily fine, but finite, precision. For a precise definition see \cite{hasselblatt_katok}, \cite{robinson}, or \cite{bowen_entropy_endos}.  The topological entropy of a geodesic or magnetic flow $\phi_t$, denoted $h_\Top(\phi_t)$, can be defined as the topological entropy of the restriction to the unit cotangent bundle of the time-1 map $\phi_1$.

\bigskip

\noindent The relationship between the topology and geometry of a Riemannian manifold and the topological entropy of its geodesic flow has been studied intensely.  In addition to the vast literature on spaces of nonpositive curvature, see \cite{dinaburg}, \cite{freire_mane}, \cite{mane}, \cite{manning}, or \cite{paternain_geod_flows} for a sampling of such results. It is interesting to ask if and how these results extend to the class of magnetic flows.  In fact, given a magnetic field $\sigma$, one can consider the family of flows corresponding to the magnetic systems $(M, g, c\sigma)$, for $c \in \mathbb{R}$, as a perturbation of the geodesic flow when $c = 0$.

\bigskip 

\noindent Some results in this direction have been obtained when the underlying geodesic flow is known to be Anosov.  G. and M. Paternain show in \cite{paternain_paternain_deriv_ent1} that if we start with an Anosov geodesic flow, and we consider a small perturbation of it by twisting the symplectic structure of $T^*M$, then the topological entropy strictly decreases.  In \cite{grognet_ent}, Grognet studies the case of compact Riemannian manifolds with negatively pinched curvature.  If the $L^\infty$ norm of the magnetic field and its first derivative are small enough relative to the curvature bounds, then the magnetic flow is Anosov and its topological entropy is bounded explicitly in terms of that of the geodesic flow.  K. Burns and G. Paternain introduce in \cite{burns_paternain_anosov_mag} a volume entropy using certain magnetic balls, and use it to bound from below the topological entropy of the magnetic flow (cf. \cite{manning}).  They also show that equality holds when the magnetic flow is Anosov.

\bigskip

\noindent In the case of Riemannian surfaces, J. Miranda shows in \cite{miranda} that if the magnetic flow has a non-hyperbolic closed orbit on some energy level $H^{-1}(c)$, then there exists a $C^\infty$-perturbation of the 2-form $\sigma$ such that the new magnetic flow has positive topological entropy in $H^{-1}(c)$. In the homogeneous setting, L. Butler and G. Paternain \cite{but_pat_sol} study the Liouville (or metric) entropy of the magnetic flow on compact quotients of the ${\bf Sol}$ Lie group determined by a left-invariant metric and a distinguished left-invariant magnetic field $\sigma$.  Modulating the intensity of $\sigma$ by a parameter $c$, they find that the entropy of the corresponding magnetic flow $\phi^c_t$ is identically $1$ for all parameter values, i.e. $h_\Top(\phi^c_t) \equiv 1$ for all $c\in [0, \infty)$.

\bigskip

\noindent In this paper, we focus on the topological entropy of left-invariant magnetic flows on 2-step nilmanifolds.  Given a simply connected nilpotent Lie group $G$ and a cocompact discrete subgroup $\Gamma < G$, the closed manifold $\Gamma \backslash G$ is called a nilmanifold.  When $G$ is $k$-step, the nilmanifold $M$ is called $k$-step.  A metric $g$ or magnetic field $\sigma$ on $M$ is called left-invariant if its pullback to $G$ is left-invariant.  Note that 1-step nilmanifolds with left-invariant metrics are precisely flat tori.  Our first result is 

\begin{thm} \label{main thm zero ent}
Let $\Gamma \backslash G$ be a 2-step nilmanifold endowed with a left-invariant metric $g$ and a left-invariant magnetic field $\sigma$ that vanishes on the derived algebra and represents a rational cohomology class.  Then the flow of the magnetic system $(\Gamma \backslash G, g, \sigma)$, on the unit cotangent bundle, has zero topological entropy.
\end{thm}

\noindent In particular, we have the following corollary.

\begin{cor} \label{main thm zero ent cor}
Let $\Gamma \backslash G$ be a 2-step nilmanifold endowed with a left-invariant metric $g$ and an exact left-invariant magnetic field $\sigma$ that represents a rational cohomology class.  Then the flow of the magnetic system $(\Gamma \backslash G, g, \sigma)$, on the unit cotangent bundle, has zero topological entropy.
\end{cor}
\noindent In contrast to \cite{butler_wild} where Butler has shown that the topological entropy of any left-invariant geodesic flow on a 2-step nilmanifold vanishes, we show that there exist a magnetic flow, with $g$ and $\sigma$ left-invariant, with positive topological entropy.  Moreover, the example we construct has positive topological entropy on arbitrarily high energy levels.  Precisely,

\begin{thm} \label{main thm pos ent}
There exists a 2-step nilmanifold $\Gamma \backslash G$ with left-invariant metric $g$ and left-invariant magnetic field $\sigma$ such that the flow of the magnetic system $(\Gamma \backslash G, g, \sigma)$ has positive topological entropy on arbitrarily high energy levels.
\end{thm}

\noindent Recall that if $\phi_t$ is the magnetic flow of a system $(M, g, \sigma)$, then for $c > 0$ the rescaled flow $\phi_{ct}$ is the magnetic flow of the system $(M,g,c \sigma)$. Hence fixing the magnetic field and varying the energy level under consideration is equivalent to modulating the field strength with a real parameter and considering the magnetic flow only on the unit cotangent bundle. Loosely speaking, the high energy limit of the magnetic flow corresponds to the underlying geodesic flow, or zero magnetic field strength.  Theorem \ref{main thm pos ent} can be rephrased: there exists arbitrarily small $c \in (0,\infty)$ such that the flow of the magnetic system $(\Gamma \backslash G, g, c \sigma)$ on the unit cotangent bundle has positive topological entropy.  Again, by Butler \cite{butler_wild}, we know that $c = 0$ yields the geodesic flow of $(\Gamma \backslash G, g)$ which has zero topological entropy.

\bigskip

\noindent We remark briefly on the relationship between topological entropy and Ma\~{n}\'{e}'s critical value in our context.  For a fixed magnetic field, the Ma\~{n}\'{e} critical value is a distinguished energy level where the dynamics of the magnetic flow change significantly. The article \cite{symp_top_mane_crit_val} is an excellent reference; the example given there of a left-invariant metric and magnetic field on a Heisenberg manifold (a nilmanifold whose universal cover is the Heisenberg group) is particularly relevant here.  For any exact magnetic field on a nilmanifold, the Ma\~{n}\'{e} critical value is finite.  Moreover, by Lemma 2.2 of \cite{burns_paternain_anosov_mag}, the citical value is positive if such a magnetic field is not identically zero.  Suppose that $(\Gamma \backslash G, g, \sigma)$ is a magnetic system satisfying the hypotheses of Corollary \ref{main thm zero ent cor}.  Since rescaling the magnetic field by a rational number does not change the fact that it represents a rational cohomology class, applying Corollary \ref{main thm zero ent cor} to the systems $(\Gamma \backslash G, g, c\sigma)$, $c \in \mathbb{Q}$, $c > 0$, shows that the topological entropy of the original system $(\Gamma \backslash G, g, \sigma)$ vanishes on energy levels both above and below the critical value.  By contrast, it is shown in \cite{symp_top_mane_crit_val} that for closed hyperbolic surfaces with magnetic field given by the area form, the topological entropy on the unit tangent bundle vanishes if and only if the field strength is below the critical value, or, equivalently, the entropy of the magnetic flow on a specific energy level vanishes if and only if the energy level is above the critical value.

\bigskip

\noindent In the example presented in Section \ref{subsec:the example}, the magnetic field is nonexact and the underlying manifold has nilpotent, and hence amenable, fundamental group.  Thus by Corollary 5.4 of \cite{paternain_mag_rig_of_hor_flows} the critical value in this case is $\infty$.  Equivalently, if we fix the energy level at one but allow the field strength to vary, then the Ma\~{n}\'{e} critical value corresponds to magnetic field strength zero.  So Theorem \ref{main thm pos ent} implies that there are field strengths arbitrarily close to the critical value zero for which the magnetic flow has positive topological entropy, while the geodesic flow (field strength zero) has zero topological entropy. Again by contrast, the particular magnetic flow on a compact quotient of ${\bf Sol}$ investigated by Butler and Paternain in \cite{but_pat_sol} also has critical value $\infty$ and has positive topological entropy on all energy levels, but unlike our example the underlying geodesic flow has positive topological entropy.

\bigskip

\noindent The main tool used in the proof is an associated principle $S^1$-bundle $P \to M$ that arises when quantizing a classical magnetic system \cite{kuwabara}.  The action of the structure group $S^1$ lifts to a Hamiltonian action on $T^*P$, and the Marsden-Weinstein reduction procedure yields a symplectic manifold that is symplectormorphic to $(T^*M, \omega_\can + \pi^* \sigma)$.  In fact, the geodesic flow on $P$ descends to a flow on the reduced manifold that is topologically conjugate to the magnetic flow.  Moreover the space $P$ in this case will be a nilmanifold. So we can combine results of Butler \cite{butler_wild}, \cite{butler_scattering} on the topological entropy of left-invariant geodesic flows on nilmanifolds with the work of Bowen \cite{bowen_entropy_endos} in order to obtain information about the topological entropy of magnetic flows on $M$.

\bigskip

\noindent In Section \ref{sec:ext_Lie_alg}, we show how the symplectic reduction procedure yields a flow conjugate to the magnetic flow and in Section \ref{sec:topological entropy} we recall some properties of topological entropy that will be needed.  In Section \ref{sec:mag flows with zero ent} we prove Theorems \ref{main thm zero ent} and \ref{main thm zero ent cor}, while in Section \ref{sec:mag flow with pos ent} we prove \ref{main thm pos ent}.

\section{The Magnetic flow via symplectic reduction} \label{sec:ext_Lie_alg}

\noindent Let $G$ be a simply connected 2-step nilpotent Lie group and let $\mathfrak{g} = \Lie(G)$ be the Lie algebra of $G$.  Given a left-invariant magnetic field $\sigma$, we can form an associated Lie group $\widetilde{G}$ as follows. Define the vector space $\widetilde{\mathfrak{g}} := \mathfrak{g} \oplus \mathbb{R}W $.  For any two vectors $X,Y \in \widetilde{\mathfrak{g}}$, let $X = X' + aW, Y = Y' + bW$ be their decomposition in terms of the direct sum, $X',Y' \in \mathfrak{g}$ and $a,b \in \mathbb{R}$.  Now define the bracket $\{ \cdot, \cdot \}$ by
\begin{align} \label{eq:extended bracket definition}
\{ X, Y \} = [X',Y'] + \sigma(X',Y')W
\end{align}
It is straightforward to check, using the fact that $\sigma$ is closed, that $\widetilde{\mathfrak{g}}$ endowed with $\{ \cdot, \cdot \}$ is a Lie algebra and, in fact, nilpotent.  Moreover, $\widetilde{\mathfrak{g}}$ is 2-step if and only if $\sigma$ vanishes on the derived algebra.  Let $\widetilde{G}$ be the simply-connected Lie group with Lie algebra $\widetilde{\mathfrak{g}}$. If $g$ is a left-invariant Riemannian metric on $G$, then we define $\widetilde{g}$ to be the left-invariant Riemannian metric such that its restriction to $\mathfrak{g}$ agrees with $g$ and $W$ is a unit vector orthogonal to $\mathfrak{g}$. 

\bigskip

\noindent Let $f:\widetilde{G} \to G$ be the surjective Lie group homomorphism whose differential at the identity is orthogonal projection $\widetilde{\mathfrak{g}} \to \mathfrak{g}$. Then it descends to a Lie group isomorphism $\bar{f}: \widetilde{G}/ \Ker(f) \to G$, which we use to identify the two Lie groups.  Also let $\Exp, \log, \widetilde{\Exp}, \widetilde{\log}$ denote the exponential maps and their inverses for $G$ and $\widetilde{G}$, respectively.

\bigskip

\noindent Next let $\Gamma < G$ be a cocompact discrete subgroup so that $\Gamma \backslash G$ is a nilmanifold.  Nomizu shows in \cite{nomizu} that real Lie algebra cohomology of $\mathfrak{g}$ is canonically isomorphic to the de Rahm cohomology of $\Gamma \backslash G$ via the map that sends an alternating  $k$-linear form $\sigma$ on $\mathfrak{g}$ to the unique left-invariant $k$-form on $G$ also denoted $\sigma$, by abuse of notation.  By a result of Sullivan \cite{sullivan}, Nomizu's result is also valid for the rational cohomologies.  In particular, a left-invariant 2-form $\sigma$ on $\Gamma \backslash G$ represents a rational cohomology class if and only if the associated bilinear form, also denoted $\sigma$, on $\mathfrak{g}$ satisfies $\sigma(\mathcal{L},\mathcal{L}) \subset \mathbb{Q}$, where $\mathcal{L} \subset \mathfrak{g}$ is the lattice given by the $\mathbb{Z}$-span of $\log(\Gamma)$. 

\bigskip

\noindent It is necessary to construct a cocompact discrete subgroup $\widetilde{\Gamma} < \widetilde{G}$ in such a way that $f(\widetilde{\Gamma}) = \Gamma$.  The existence of such a subgroup is equivalent to the condition that $\sigma$ represents a rational cohomology class.  Since there is an integer $k \in \mathbb{Z}$ such that $k \sigma(\mathcal{L}, \mathcal{L}) \subset \mathbb{Z}$, we can form the set $S = \{ X + n/(12k^2)W | X \in \log(\Gamma), n \in \mathbb{Z} \} \subset \widetilde{\mathfrak{g}}$.  It is straightforward to check using the Baker-Campbell-Hausdorff formula and \eqref{eq:extended bracket definition} that $\widetilde{\Gamma} := \Exp(S)$ is the desired subgroup.  Note that there are many cocompact, discrete subgroups $\widetilde{\Gamma} < \widetilde{G}$ satisfying the condition $f(\widetilde{\Gamma}) = \Gamma$, but the choice does not affect what follows.

\bigskip

\noindent The Lie group homomorphism $f: \widetilde{G} \to G$ now descends to a smooth map $\widetilde{\Gamma} \backslash \widetilde{G} \to \Gamma \backslash G$.  This map will also be denoted as $f(\widetilde{\Gamma}x) = \Gamma f(x)$ where $x \in \widetilde{G}$.  Now $\widetilde{\Gamma} \backslash \widetilde{G}$ is the total space of a principal $S^1$-bundle over $\Gamma \backslash G$.  The action of the structure group is given by $\rho_t : \widetilde{\Gamma} \backslash \widetilde{G} \to \widetilde{\Gamma} \backslash \widetilde{G}$ by
\begin{align} \label{eq:princpal S^1 action}
\rho_t(x) = L_{\widetilde{\Exp}(tW)}(x) = \widetilde{\Exp}(tW)x
\end{align}
The action lifts to a Hamiltonian action on the cotangent bundle $T^*(\widetilde{\Gamma} \backslash \widetilde{G})$ with its canonical symplectic structure $\widetilde{\omega}_\can$ and with moment map $\mu : T^* (\widetilde{\Gamma} \backslash \widetilde{G}) \to \mathbb{R}$ given by $\mu(\widetilde{\Gamma}x,\lambda) = \lambda(W)$.  The differential $d\mu$ never vanishes, so $\mu^{-1}(c)$ is a submanifold for every $c \in \mathbb{R}$.  Moreover the action restricted to any level set is free.  Hence the Marsden-Weinstein reduction procedure yields a symplectic manifold $(\mu^{-1}(c)/S^1, \widetilde{\omega}_\red)$. The action also preserves the metric Hamiltonian $\widetilde{H} : T^*(\widetilde{\Gamma} \backslash \widetilde{G}) \to \mathbb{R}$, given by $\widetilde{H}(x,\lambda) = \widetilde{g}_x(\lambda, \lambda)/2$, and so induces a function $\bar{H}_c : \mu^{-1}(c)/S^1 \to \mathbb{R}$. In this situation, Kuwabara shows the following.
\begin{prop} [\cite{kuwabara}, Proposition 1.1] \label{kuwabara summary}
For each $c \in \mathbb{R}$ there exists a diffeomorphism $\psi_c : \mu^{-1}(c)/S^1 \to T^*(\Gamma \backslash G)$ such that
\begin{align*}
\widetilde{\omega}_\red = \psi^*_c(\omega_\can + \pi^*(c\sigma)) \qquad \qquad \bar{H}_c = \psi^*_c H + \frac{c^2}{2}
\end{align*}
where $H: T^*(\Gamma \backslash G) \to \mathbb{R}$ is the metric Hamiltonian.  In particular, the Hamiltonian system $(T^*(\Gamma \backslash G), \omega_\can + \pi^*(c\sigma), H)$, which yields the magnetic flow of $(\Gamma \backslash G, g, c \sigma)$, is isomorphic with $(\mu^{-1}(c)/S^1, \widetilde{\omega}_\red, \bar{H}_c)$, and taking $c = 0$ yields the geodesic flow on $\Gamma \backslash G$.
\end{prop}

\noindent We now focus on specific energy levels since they are invariant compact subsets.  For any positive real numbers $c, d > 0$, define the set
\begin{align} \label{eq:def of the invariant set}
\Lambda_{c,d} = \mu^{-1}(c) \cap \widetilde{H}^{-1}\bigg( \frac{d^2 + c^2}{2}\bigg) \subset T^*(\widetilde{\Gamma} \backslash \widetilde{G})
\end{align}

\begin{lem} \label{specific energy levels}
The set $\Lambda_{c,d}$ is invariant under both the $S^1$ action $\rho_t$ and the geodesic flow $\widetilde{\phi}_t$ on $T^*(\widetilde{\Gamma} \backslash \widetilde{G})$.  Moreover, $\Lambda_{c,d} / S^1 = \bar{H}^{-1}_c((d^2 + c^2)/2)$ and the geodesic flow descends to a flow $\bar{\phi}_t$ on $\Lambda_{c,d} / S^1$.
\end{lem}

\begin{proof}
Since both actions are Hamiltonian, they preserve their respective Ham- \allowbreak iltonian functions, $\widetilde{H}$ and $\mu$.  Moreover the $\rho_t$ action is the lift of an isometric action and hence preserves $\widetilde{H}$.  It follows that $\widetilde{H}$ and $\mu$ Poisson commute and are preserved under each others Hamiltonian flows. Thus $\widetilde{H}$ and $\mu$, and hence the set $\Lambda_{c,d}$, are invariant under both the geodesic flow $\widetilde{\phi_t}$ and the $S^1$ action $\rho_t$.

\bigskip

\noindent Let $q: \Lambda_{c,d} \to \Lambda_{c,d} / S^1$ denote the quotient map.  Then for any $p \in \mu^{-1}(c)$, 
\begin{align*}
q(p) \in \bar{H}_c^{-1}((d^2+c^2)/2) \iff \widetilde{H}(p) = \frac{d^2+c^2}{2} \iff p \in \Lambda_{c,d}
\end{align*}
showing that $\Lambda_{c,d} / S^1 = \bar{H}^{-1}((d^2+c^2)/2)$.

\bigskip

\noindent For the last statement, note that because the action $\rho_t$ comes from an action by isometries (left multiplication) on the base manifold, it preserves the geodesic vector field.
\end{proof}

\noindent An immediate consequence of Proposition \ref{kuwabara summary} and Lemma \ref{specific energy levels} is the following.

\begin{thm} \label{symp_quot_id_with_mag_system}
The map $\psi_c$ in Proposition \ref{kuwabara summary} conjugates the Hamiltonian flow of $(\mu^{-1}(c)/S^1, \widetilde{\omega}_\red, \bar{H})$ restricted to the energy level $\bar{H}^{-1}((d^2 + c^2)/2)$ with the Hamiltonian flow of $(T^*(\Gamma \backslash G), \omega_\can + c \pi^* \sigma, H)$ restricted to the energy level $H^{-1}(d^2/2)$.  In particular, the induced flow on $\Lambda_{c,d} / S^1$ is topologically conjugate to the magnetic flow of $(\Gamma \backslash G, g, c \sigma)$ on the energy level $H^{-1}(d^2/2)$.
\end{thm}

\section{Topological entropy} \label{sec:topological entropy}

We collect some results on topological entropy that will be used below.  The topological entropy of a flow $\phi_t$ on a compact metric space $(X,d)$ can be defined as the topological entropy of the time-1 map $\phi_1:X \to X$.     A couple of basic properties of the topological entropy are the following.

\begin{thm}[\cite{hasselblatt_katok}, Proposition 3.1.7] \label{top_ent_basic_props}
Let $(X,d)$ be a compact metric space and $f:X \to X$ a continuous map.  Then
\begin{enumerate}
\item If $\Lambda \subset X$ is a closed $f$-invariant set, then $h_\Top(f) \geq h_\Top(f|_\Lambda)$.
\item If $(Y,e)$ is another compact metric space, $g : Y \to Y$ a continuous map, and $h: X \to Y$ a continuous surjective map such that $h \circ f = g \circ h$, then $h_\Top(f) \geq h_\Top(g)$.
\end{enumerate}
\end{thm}

\noindent Changing the speed of a flow, scales the topological entropy by the same factor.

\begin{thm} [\cite{bowen_entropy_endos}, Proposition 21] \label{top_ent_prop_1}
For any $c \in \mathbb{R}$, $\phi_{ct}$ is again a flow, and $h_\Top(\phi_{ct}) = |c| h_\Top(\phi_t)$.
\end{thm}

\noindent As a corollary, we have the following.

\begin{cor}[\cite{paternain_geod_flows}] \label{changing_energy_geod}
If the geodesic flow on the unit tangent bundle of a compact Riemannian manifold $(M,g)$ has zero (resp. positive) topological entropy, then the geodesic flow restricted to any nonzero energy level has zero (resp. positive) topological entropy.  More precisely, if $\phi_t$ is the geodesic flow of $(M,g)$ and $\psi_t$ is the geodesic flow of $(M,cg)$, then
\begin{align*}
h_\Top(\psi_t) = \frac{h_\Top(\phi_t)}{\sqrt{c}}
\end{align*}
\end{cor}

\noindent The next two results were proved by Bowen.

\begin{thm}[\cite{bowen_entropy_endos}, Corollary 18] \label{bowen_sup_inv_levels}
Let $(X,d)$ and $(Y,e)$ be compact metric spaces and $f:X \to X$ and $\pi:X \to Y$ be continuous with $\pi \circ f = \pi$.  Then
\begin{align*}
h_\Top(f) = \sup_{y \in Y} h_\Top(f|_{\pi^{-1}(y)})
\end{align*}
\end{thm}

\begin{thm}[\cite{bowen_entropy_endos}, Theorem 19] \label{bowen_quotients}
Let $(X,d)$ be a compact metric space and $f:X \to X$ a continuous map. Let $(G, d)$ be a compact topological group that acts freely.  Let $\pi :X \to Y$ be the orbit map.  If $f$ is $G$-equivariant, and $g:Y \to Y$ is the map induced by $f$, i.e. $g \circ \pi = \pi \circ f$, then
\begin{align*}
h_\Top(f) = h_\Top(g)
\end{align*}
\end{thm}

\section{Magnetic Flows with Zero Topological Entropy} \label{sec:mag flows with zero ent}

\noindent Let $G$ be a simply connected 2-step nilpotent Lie group, $g$ a left-invariant metric on $G$, $\Gamma < G$ a cocompact discrete subgroup, and $\sigma$ a left-invariant magnetic field that vanishes on the derived algebra and represents a rational cohomology class.

\bigskip

\noindent Butler shows that the geodesic flow of any two-step nilmanifold has zero topological entropy.  More precisely, 
\begin{thm}[\cite{butler_wild}, Theorem 1.4] \label{butler_zero_ent}
Let $K$ be a connected, simply connected two-step nilpotent Lie group, and $\Gamma \leq K$ be a discrete, cocompact subgroup of $K$.  If $g$ is a left-invariant metric on $K$ and $\Phi_t$ is the geodesic flow induced by $g$ on $T^*(\Gamma \backslash K)$, then
\begin{align*}
h_\Top(\Phi_t) = 0
\end{align*}
\end{thm}

\noindent Using this result we can construct magnetic systems on 2-step nilmanifolds with zero topological entropy.

\begin{proof}[Proof of Theorem \ref{main thm zero ent}]
Use the magnetic field $\sigma$ to form the Lie group $\widetilde{G}$ and cocompact discrete subgroup $\widetilde{\Gamma} < \widetilde{G}$, as described in Section \ref{sec:ext_Lie_alg}.  Since $\sigma$ vanishes on the derived algebra, $\widetilde{G}$ is a 2-step nilpotent Lie group.  Let $\widetilde{H} : T^*(\widetilde{\Gamma} \backslash \widetilde{G}) \to \mathbb{R}$ be the metric Hamiltonian, and let $\widetilde{\phi}_t$ denote the geodesic flow on $T^*(\widetilde{\Gamma} \backslash \widetilde{G})$.

\bigskip

\noindent Consider the set
\begin{align*}
\Lambda := \Lambda_{c,1} = \widetilde{H}^{-1}\bigg( \frac{1+c^2}{2} \bigg) \cap \mu^{-1}(c) \subset T^*(\widetilde{\Gamma} \backslash \widetilde{G})
\end{align*}
as defined in equation \eqref{eq:def of the invariant set}.  Theorem \ref{butler_zero_ent} implies that the geodesic flow $\widetilde{\phi}_t$ has zero topological entropy.  By Corollary \ref{changing_energy_geod}, the topological entropy restricted to any energy level in $T^*(\widetilde{\Gamma} \backslash \widetilde{G})$ vanishes.  In particular, $h_\Top(\widetilde{\phi}_t|_{\widetilde{H}^{-1}((1+c^2)/2)}) = 0$.  Let $\bar{\phi}_t$ be the flow induced by the geodesic flow on $\Lambda /S^1$.  Since $S^1$ is a compact group acting on $\Lambda$, and $\Lambda$ is an invariant set,
\begin{align*}
h_\Top(\bar{\phi}_t) = h_\Top(\widetilde{\phi}_t|_\Lambda) \leq h_\Top(\widetilde{\phi}_t|_{\widetilde{H}^{-1}((1+c^2)/2)}) = 0
\end{align*}
where the first equality follows from Theorem \ref{bowen_quotients} and the inequality follows from Theorem \ref{top_ent_basic_props}.  To finish the proof take $c = 1$ and note that, by Theorem \ref{symp_quot_id_with_mag_system}, the flow $\bar{\phi}_t$ on $\Lambda / S^1$ is conjugate to the flow of the magnetic system $(\Gamma \backslash G, g, \sigma)$ on $H^{-1}(1/2)$.

\end{proof}

\noindent We now prove Corollary \ref{main thm zero ent cor}.

\begin{proof}[Proof of Corollary \ref{main thm zero ent cor}]
As noted in Section \ref{sec:ext_Lie_alg}, the Lie algebra cohomology \linebreak $H^*(\mathfrak{g}; \mathbb{R})$ is isomorphic to de Rahm cohomolgy  $H^*(\Gamma \backslash G; \mathbb{R})$.  If $d$ is the boundary map for Lie algebra cohomology, $\tau \in \Lambda^1(\mathfrak{g}) \simeq \mathfrak{g}^*$ is a linear functional, then for any $X \in \mathfrak{g}$ and any $D \in [\mathfrak{g}, \mathfrak{g}]$,
\begin{align*}
d\tau(X,D) = - \tau([X,D]) = 0
\end{align*}
Hence any exact left-invariant magnetic field must vanish on the derived algebra, and the corollary follows from Theorem \ref{main thm zero ent}.
\end{proof}

\section{Example of a Magnetic Flow with Positive Topological Entropy} \label{sec:mag flow with pos ent}

\noindent In this section, we show that there exists a magnetic flow on a 2-step nilmanifold with positive topological entropy.  The essential difference from the situation in Theorem \ref{main thm zero ent} is that the magnetic field does \emph{not} vanish on the derived algebra.  Hence the associated Lie group $\widetilde{G}$ is now 3-step.  

\subsection{Reduction of Hamiltonian flow to Euler flow} \label{subsec:ham to euler}

\noindent This section follows the exposition of Section 2.1 in \cite{but_gel}. Recall that for any Lie group $G$ and any left-invariant Hamiltonian $H:T^*G \to \mathbb{R}$, the Hamiltonian flow projects to a Hamiltonian flow on the Poisson manifold $\mathfrak{g}^*$ as follows. First, the dual to a Lie algebra, $\mathfrak{g}^*$ is naturally a Poisson manifold as follows.  For any $f \in C^\infty(\mathfrak{g}^*)$, the value of $df$ at a point $\lambda \in \mathfrak{g}^*$ can be thought of as a vector in the Lie algebra $\mathfrak{g}$ via the sequence of identifications $df_\lambda \in T^*_\lambda \mathfrak{g}^* = (\mathfrak{g}^*)^* = \mathfrak{g}$.  The Poisson bracket for $f,k \in C^\infty(\mathfrak{g}^*)$ is now defined as
\begin{align*}
\{f,k\}(\lambda) = - \lambda([df_\lambda, dk_\lambda])
\end{align*}
Given $f \in C^\infty(\mathfrak{g}^*)$, then $E_f(\cdot) = \{ \cdot , f \}$ is a vector field on $\mathfrak{g}^*$, called the Euler vector field $E_f$ associated to $f$.  Using the identifications $T^*_\lambda \mathfrak{g}^* = (\mathfrak{g}^*)^* = \mathfrak{g}$ and $T_\lambda \mathfrak{g}^* = \mathfrak{g}^*$, and the natural pairing $\langle \cdot, \cdot \rangle$ between $\mathfrak{g}$ and $\mathfrak{g}^*$, we have for any function $\varphi \in C^\infty(\mathfrak{g}^*)$
\begin{align*}
\langle d\varphi_\lambda, (E_f)_\lambda \rangle &= \{ \varphi, f \}(\lambda) = - \langle \lambda, [d\varphi_\lambda, df_\lambda] \rangle = \langle \lambda \circ \ad_{df_\lambda}, d\varphi_\lambda \rangle
\end{align*}
Since the function $\varphi$ was arbitrary, we get the explicit expression for the Euler vector field
\begin{align} \label{eqn:ham_vf_on_lie_dual}
(E_f)_\lambda = \lambda \circ \ad_{df_\lambda}
\end{align}

\bigskip

\noindent Next, trivialize the cotangent bundle $T^*G \simeq G \times \mathfrak{g}^*$ via left-invariant 1-forms, $\lambda \in T_g^*G \mapsto (g, L_g^*\lambda)$.  The tangent bundle of $T^*G$ is now identified with the tangent bundle of $G \times \mathfrak{g}^*$, by
\begin{align*}
T(T^*G) = T(G \times \mathfrak{g}^*) = (G \times \mathfrak{g}^*) \times (\mathfrak{g} \oplus \mathfrak{g}^*)
\end{align*}
where the identification $T_gG = \mathfrak{g}$ via left-invariant vector fields is implicit.  With these conventions, the Liouville 1-form is given by
\begin{align*}
\theta_{(g,\lambda)}(X,\alpha) = \lambda(X)
\end{align*} 
and the symplectic form is
\begin{align*}
\omega_{(g,\lambda)}((X,\alpha),(Y,\beta)) &= -d\theta_{(g,\lambda)}((X,\alpha),(Y,\beta)) = \beta(X) - \alpha(Y) + \lambda([X,Y])
\end{align*}
Let $r:G \times \mathfrak{g}^* \to \mathfrak{g}^*$ be projection onto the second factor, $f \in C^\infty(\mathfrak{g}^*)$, and $r^*f = f \circ r$ its pullback to the cotangent bundle.  One can calculate now that the Hamiltonian vector field of $r^*f$ is given by
\begin{align} \label{eqn:hamiltonian_vf_cot_bund}
X_{r^*f}(g,\lambda) = (df_\lambda, \lambda \circ \ad_{df_\lambda})
\end{align}

\bigskip

\noindent In conclusion, we have

\begin{claim}[cf. \cite{but_gel}, Section 2.1]
For any $f \in C^\infty(\mathfrak{g}^*)$, the map $r : G \times \mathfrak{g}^* \to \mathfrak{g}^*$ projects the Hamiltonian flow of $r^*f$ onto the Euler flow of $f$. 
\end{claim}

\begin{proof}
Combine equations \eqref{eqn:hamiltonian_vf_cot_bund} and \eqref{eqn:ham_vf_on_lie_dual}.
\end{proof}

\noindent As a particular case, we have the following the corollary.

\begin{cor} \label{geod flow projects to euler flow}
The geodesic flow of any left-invariant metric on $G$ projects to an Euler flow on $\mathfrak{g}^*$.
\end{cor}

\begin{proof}
Any inner product $\langle \cdot, \cdot \rangle$ on $\mathfrak{g}^*$ determines a left-invariant metric on $G$.  If $f$ is the function on $\mathfrak{g}^*$ given by $f(\lambda) = (1/2)\langle \lambda, \lambda \rangle$, then $r^*f$ is the Hamiltonian that yields the geodesic flow of the left-invariant metric on $T^*G$.
\end{proof}

\subsection{Horseshoes}

For each natural number $N \geq 2$ define the space of bi-infinite sequences on $N$ symbols as
\begin{align*}
\Omega_N = \{ \omega = (\ldots, \omega_{-1}, \omega_0, \omega_1, \ldots ) \ | \ \omega_i \in \{0,1,\ldots,N-1 \} \text{ for } i \in \mathbb{Z} \}
\end{align*}
For any fixed $\lambda > 1$, define a metric on this set by
\begin{align*}
d_\lambda(\omega, \omega') = \sum_{n = -\infty}^\infty \frac{|\omega_n - \omega'_n|}{\lambda^{|n|}}
\end{align*}
The topology induced by this metric on $\Omega_N$ does not depend on the choice of $\lambda$.  Define the left-shift $\sigma_N : \Omega_N \to \Omega_N$ by
\begin{align*}
\sigma_N(\omega) = \omega' = (\ldots, \omega'_0, \omega'_1, \ldots)
\end{align*}
where $\omega'_i = \omega_{i+1}$.  It is a homeomorphism.  Let $A = (a_{ij})_{i,j = 0}^{N-1}$ be any matrix whose entries are either $0$ or $1$.  Then one can form the subset
\begin{align*}
\Omega_A = \{ \omega \in \Omega_N \ | \ a_{\omega_n \omega_{n+1}} = 1 \text{ for }n \in \mathbb{Z} \}
\end{align*}
In other words, the matrix $A$ determines the allowable transitions.  When $A$ is the matrix of all $1$'s, $\Omega_A = \Omega_N$.  Further, the set $\Omega_A$ is shift invariant, so we can consider the restriction $\sigma_A = \sigma_N|_{\Omega_A}$, called a topological Markov chain, or a subshift of finite type.  A matrix $A$ is called transitive if for some positive $m$, all the entries of $A^m$ are positive.  Recall that a dynamical system is called topologically transitive if there exists a dense orbit.  A topological Markov chain $\sigma_A$ is called transitive if $A$ is a transitive matrix.  The terminology is motivated by the following proposition:

\begin{prop}[\cite{hasselblatt_katok}, Proposition 1.9.9]
If $A$ is a transitive matrix, then the topological Markov chain $\sigma_A$ is topologically mixing and its periodic orbits are dense in $\Omega_A$.
\end{prop}

\begin{cor} \label{trans_of_horseshoe}
If $A$ is a transitive matrix, then the topological Markov chain $\sigma_A$ is topologically transitive.
\end{cor}

\begin{proof}
The space $\Omega_A$ is a locally compact, seperable, metric space.  Thus by Lemma 1.4.2 of \cite{hasselblatt_katok}, $\sigma_A$ is topologically transitive if and only if it topologically mixing.
\end{proof}

\noindent Robinson gives the following definition of a horseshoe (pg 287-288 of \cite{robinson}).  A diffeomorphism $f$ is said to have a horseshoe, if there exists an invariant set $\Theta$ (or subsystem) for which $f|_\Theta$ is topologically conjugate to the shift map $\sigma_A$ on a transitive two-sided subshift of finite type on $n$ symbols.

\bigskip

\noindent In order to be able to describe the same sort of chaotic dynamics for flows, we need to extend definitions slightly.  Given a metric space $X$, form the product $X \times \mathbb{R}$ and define a flow $\varphi_t(x,s) = (x,s+t)$.  Given a homeomorphism $f:X \to X$ and a continuous function $\tau: X \to (0,\infty)$, define an equivalence relation on $X \times \mathbb{R}$ by $(x,s+\tau(x)) \sim (f(x),s)$.  Finally, the flow descends to a flow $\varphi^\tau_t$ on $\hat{X} = (X \times \mathbb{R}) / \sim$.  In the case that $X = \Omega_A$ and $f = \sigma_A$ for some transitive $A$, this system is called a suspended horseshoe.  A flow which possesses an invariant set that is topologically conjugate to $\varphi^\tau_t$ is said to have a suspended horseshoe.

\bigskip

\begin{lemma} \label{trans_of_suspension}
Let $f:X \to X$ be a topologically transitive homeomorphism of a metric space.  Then for any continuous $\tau : X \to (0, \infty)$, the time-variable suspension flow $\varphi^\tau_t : \hat{X} \to \hat{X}$ is also topologically transitive.
\end{lemma}

\begin{proof}
Since $f$ is topologically transitive, there exists $x \in X$ such that the orbit of $x$ under $f$ is dense.  Let $U \subset \hat{X}$ be any open set, and let $q: X \times \mathbb{R} \to \hat{X}$ be the quotient map.  Since open rectangles are a base for the topology on $X \times \mathbb{R}$, we can find an open rectangle $V \times W \subset q^{-1}(U)$. By the density of the orbit of $x$, there exists a natural number $n$ such that $f^n(x) \in V$.  Set $T_1 = \sum_{i=0}^{n-1} \tau(f^i(x))$ and fix $T_2 \in W$. If $[ (x,t) ]$ denotes the equivalence class of $(x,t)$ in the quotient $\hat{X}$, then we have the identity
\begin{align*}
\varphi^\tau_{\tau(f^i(x))}[(f^i(x),0)] &= [\varphi_{\tau(f^i(x))}(f^i(x),0)]  \\
&= [(f^i(x),\tau(f^i(x)))] \\
&= [(f^{i+1}(x), 0)]
\end{align*}
Using this identity
\begin{align*}
\varphi^\tau_{T_1 + T_2}[(x,0)] &= \varphi^\tau_{T_2} \circ \varphi^\tau_{\tau(f^{n-1}(x))} \circ \cdots \circ \varphi^\tau_{\tau(x)}[(x,0)] \\
&= \varphi^\tau_{T_2}[(f^n(x),0)] \\
&= [(f^n(x),T_2)] \in q(V \times W) \subset U
\end{align*}
\end{proof}

\noindent A consequence of Lemma \ref{trans_of_suspension} is the following.

\begin{lemma} \label{horseshoe contained in energy sphere}
Let $\varphi_t : X \to X$ be a flow possessing a suspended horseshoe $\Theta \subset X$, and suppose that $\pi: X \to \mathbb{R}$ is a continuous function invariant under the flow, $\pi \circ \varphi_t = \pi$.  Then $\Theta \subset \pi^{-1}(c)$ for some $c \in \mathbb{R}$.
\end{lemma}

\begin{proof}
By Corollary \ref{trans_of_horseshoe} and Lemma \ref{trans_of_suspension}, and the fact that topological transitivity is an invariant of topological conjugacy, $\varphi_t |_\Lambda$ is topologically transitive.  Fix $x \in \Theta$ such that $x$ has a dense orbit in $\Theta$ and set $c = \pi(x)$.  The lemma now follows. 
\end{proof}

\subsection{The Example} \label{subsec:the example}

\noindent Let $\mathfrak{g} = \Span_\mathbb{R} \{ U,V,X,Y,Z \}$.  Define a Lie bracket on the basis vectors by $[X,Y] = Z$ and $[Y,V] = U$ with all other pairs being zero, and extend $[\cdot, \cdot]$ by linearity to all of $\mathfrak{g}$.  Let $G$ be the simply connected Lie group with $\mathfrak{g} = \Lie(G)$, and endow $G$ with the left-invariant metric, $g$, such that $\{ U,V,X,Y,Z \}$ is an orthonormal basis for $T_e G$.  For magnetic field, let $\sigma$ be the left-invariant 2-form on $G$ such that at the identity $\sigma(X,U) = \sigma(Z,V) = 1$.  Fix any cocompact discrete subgroup $\Gamma < G$ such that $\sigma$ represents a rational cohomology class with respect to $\Gamma$.  After performing the extension procedure outlined in Section \ref{sec:ext_Lie_alg}, note that $\widetilde{G}$ and $\widetilde{\mathfrak{g}}$ are the same as $T_4$ and $\mathfrak{t}_4$ from \cite{butler_scattering}. 

\bigskip

\noindent Following the notation and terminology of \cite{butler_scattering}, introduce coordinates $( p_U, p_V, p_W,\allowbreak p_X, p_Y, p_Z )$ on $\widetilde{\mathfrak{g}}^*$ where $p_A(\lambda) = \lambda(A)$ for $A \in \{ U,V,W,X,Y,Z \}$.  Note that $p_W \circ r(x,\lambda) = \lambda(W) = \mu(x,\lambda)$, where $r: \widetilde{G} \times \widetilde{\mathfrak{g}}^* \to \widetilde{\mathfrak{g}}^*$ is the projection map introduced in Section \ref{subsec:ham to euler}, $W$ is the distinguished vector, and $\mu: T^*\widetilde{G} \to \mathbb{R}$ is the moment map introduced in Section \ref{sec:ext_Lie_alg}.  The functions $K_1 = p_W$ and $K_2 = p_W p_Y - p_Z p_U$ are the two independent Casimirs of $\widetilde{\mathfrak{g}}^*$, and the level sets of the function $K = (K_1,K_2) : \widetilde{\mathfrak{g}}^* \to \mathbb{R}^2$ are the coadjoint orbits. Let $\mathcal{O}_{(k_1,k_2)}$ be the coadjoint orbit $K^{-1}(k_1,k_2)$.  A coadjoint orbit $\mathcal{O}_{(k_1,k_2)}$ will be called regular if $k_1 k_2 \neq 0$. Let $\widetilde{h}: \widetilde{g}^* \to \mathbb{R}$, defined by $\widetilde{h} \circ r = \widetilde{H}$, where $\widetilde{H}: \widetilde{\Gamma} \backslash \widetilde{G} \times \widetilde{\mathfrak{g}}^* \to \mathbb{R}$ is the induced metric Hamiltonian, and $r$ is projection onto the second factor.  Then $\widetilde{h}$ is precisely the Hamiltonian defined in Equation (5) of Section 2.3 of \cite{butler_scattering} with $a_{ij} = 2$ for $1 \leq i < j \leq 4$. Finally, define $\alpha \in \mathbb{C}$ by
\begin{align} \label{eqn:defn of butlers alpha invariant}
\alpha = \sqrt{\frac{k_2 - k_1^2}{k_2 + k_1^2}}
\end{align}
which depends only on the coadjoint orbit.  This is the quantity $\alpha$ in Equation (7) of \cite{butler_scattering} for the specific Hamiltonian $\widetilde{h}$. We now conclude that
\begin{thm}[\cite{butler_scattering}, Theorem 1.2] \label{butler mainthm from scattering}
Let $\mathcal{O}_{(k_1,k_2)} \subset \widetilde{\mathfrak{g}}^*$ be a coadjoint orbit for which $\alpha \in \mathbb{R}$, $\alpha \neq 0$ (as defined in \eqref{eqn:defn of butlers alpha invariant}).  Then the Euler vector field restricted to $\mathcal{O}_{(k_1,k_2)}$ has a horseshoe.
\end{thm}
\noindent Although the statement of Theorem \ref{butler mainthm from scattering} is slightly different than that of Theorem 1.2 in \cite{butler_scattering}, this is what the paper proves and will be necessary below.  We now use this to prove Theorem \ref{main thm pos ent}.

\begin{proof}[Proof of Theorem \ref{main thm pos ent}]
\noindent Fix $k_1 = 1$.  For any $D > 1$, let $k_2 > 0$ such that $\sqrt{2k_2 -1} = D$.  Now $\alpha$ is real and positive, so, by Theorem \ref{butler mainthm from scattering}, $\mathcal{O}_{(1,k_2)}$ contains a horseshoe, $\Theta$. By Lemma \ref{horseshoe contained in energy sphere}, there exists a constant $b \in \mathbb{R}$ such that $\Theta \subset \widetilde{h}^{-1}(b)$.  Note also that $\widetilde{h} \geq K_2$, so $b \geq k_2$, and that $K_1$ induces the moment map $\mu$, i.e. $K_1 \circ r = \mu$.

\bigskip

\noindent Consider the set $\Lambda_{1,d}$, defined in \eqref{eq:def of the invariant set}, where $d = \sqrt{2b - 1}$.  For any $\lambda \in \Theta$, let $(x,\lambda) \in r^{-1}(\lambda)$.  Then
\begin{align*}
\lambda \in \Theta 
\Longrightarrow
\begin{array}{l}
 K_1(\lambda) = 1 \\
 \widetilde{h}(\lambda) = b
\end{array}  
\Longrightarrow
\begin{array}{l}
\mu(x,\lambda) = 1 \\
\widetilde{H}(x,\lambda) = b
\end{array} 
\Longrightarrow
(x,\lambda) \in \Lambda_{1,d}
\end{align*}
Hence $\Theta \subset r(\Lambda_{1,d})$.  The presence of a horseshoe in the invariant set $r(\Lambda_{1,d})$ implies that the Euler flow of $\widetilde{h}$ restricted to $ r(\Lambda_{1,d})$ has positive topological entropy. Since $r$ projects the Hamiltonian flow $\widetilde{H}$ onto the Euler flow of $\widetilde{h}$, we have by part (2) of Theorem \ref{top_ent_basic_props} that $h_\Top(\widetilde{\phi}_t) > 0$ where $\widetilde{\phi}_t$ is the restriction of the geodesic flow to the invariant set $\Lambda_{1,d}$.

\bigskip

\noindent Let $\psi_t$ be the magnetic flow of $(\Gamma \backslash G,g,\sigma)$ on the energy level $H^{-1}(d^2/2) \subset T^*(\Gamma \backslash G)$, and let $\phi_t$ be the flow induced by $\widetilde{\phi}_t$ on the symplectically reduced space $\mu^{-1}(1)/S^1$.  By Theorem \ref{bowen_quotients} the induced flow has the same topological entropy as $\widetilde{\phi}_t$, and by Theorem \ref{symp_quot_id_with_mag_system}, $\psi_t$ and $\phi_t$ are topologically conjugate.  Thus
\begin{align*}
h_\Top(\psi_t) = h_\Top(\phi_t) = h_\Top(\widetilde{\phi}_t) > 0
\end{align*}
Lastly, note that
\begin{align*}
d = \sqrt{2b - 1} \geq \sqrt{2k_2 - 1} = D
\end{align*}
Since $d^2 / 2$ is the energy level of the magnetic flow, and $D$ was arbitrary, this proves the theorem.
\end{proof}

\noindent {\bf Remark.} Suppose that $\gamma$ is a magnetic geodesic for the system $(M,g,\sigma)$ with $|\gamma'(t)| \equiv d$. Then $\gamma(t/d)$ is a unit speed magnetic geodesic for the system $(M,g,\allowbreak(1/d)\sigma)$.  So Theorem \ref{main thm pos ent} can be rephrased as saying the topological entropy of $(\Gamma \backslash G, g, c\sigma)$ on the unit cotangent bundle has positive topological entropy for arbitrarily small values of $c$.  However, for $c = 0$, the topological entropy vanishes \cite{butler_wild}.

\section*{Acknowledgments}

\noindent The author would like to thank Carolyn Gordon for all her help and guidance throughout the research process and the preparation of this paper.  The author is also grateful to Leo Butler for comments on an earlier version of this paper.

\bibliography{symp_red_on_2-step}{}
\bibliographystyle{amsplain}

\end{document}